\documentclass[reqno]{amsart}

\usepackage[utf8]{inputenc}
\usepackage[T1]{fontenc}
\usepackage{amsthm}
\usepackage{amsmath}
\usepackage{amssymb}
\usepackage{comment}
\usepackage[inline]{enumitem}
\usepackage[dvips]{graphicx}
\usepackage{color}
\usepackage{hyperref}
\usepackage{url}
\usepackage{fancyhdr}
\usepackage{mathrsfs}
\usepackage{stmaryrd}
\usepackage[normalem]{ulem}

\newtheorem{theorem}{Theorem}[section]
\newtheorem{lemma}[theorem]{Lemma}
\newtheorem{proposition}[theorem]{Proposition}

\theoremstyle{definition}

\theoremstyle{remark}

\newtheorem*{claim*}{\textsc{Claim}}

\hypersetup{
    pdftitle={On the density of sumsets},
    pdfauthor={Paolo Leonetti and Salvatore Tringali},
    pdffitwindow=true,
    pdfstartview=FitH,
    colorlinks=true,
    linkcolor=blue,
    citecolor=green,
    urlcolor=cyan
}

\providecommand{\dom}{{\rm dom}}
\providecommand{\HHb}{\mathbf{H}}

\providecommand{\NNb}{\mathbf{N}}

\providecommand{\RRb}{\mathbf{R}}

\providecommand\llb{\llbracket}
\providecommand\rrb{\rrbracket}

\renewcommand{\bf}{\mathbf}
\renewcommand{\emptyset}{\varnothing}
\renewcommand{\rho}{\varrho}
\renewcommand{\emph}[1]{\textsf{#1}}

\begin{document}

\title{On the Density of Sumsets}

\author{Paolo Leonetti}
\address{Institute of Analysis and Number Theory, Graz University of Technology | Kopernikusgasse 24/II, 8010 Graz, Austria}
\curraddr{Department of Decision Sciences, Universit\`a ``Luigi Bocconi'' | via Roentgen 1, 20136 Milano, Italy} 
\email{leonetti.paolo@gmail.com}
\urladdr{https://sites.google.com/site/leonettipaolo/}
\author{Salvatore Tringali}
\address{School of Mathematics,
	Hebei Normal University | Shijiazhuang, Hebei province, 050024 China}
\email{salvo.tringali@gmail.com}
\urladdr{http://imsc.uni-graz.at/tringali}

\thanks{P.L. was supported by the Austrian Science Fund (FWF), project F5512-N26 and by PRIN 2017, grant 2017CY2NCA.}
%
%
\begin{abstract}
\noindent{}
Recently introduced by the authors in [Proc.~Edinb.~Math.~Soc.~\textbf{60} (2020), 139--167], quasi-densities form a large family of real-valued functions partially defined on the power set of the integers that 
serve as a unifying framework for the study of many known densities (including the asymptotic density, the Banach density, the logarithmic density, the analytic density, and the P\'{o}lya density). 

We further contribute to this line of research by proving that (i) for each $n \in \mathbf N^+$ and $\alpha \in [0,1]$, there is $A \subseteq \mathbf{N}$ with $kA \in \text{dom}(\mu)$ and $\mu(kA) = \alpha k/n$ for every quasi-density $\mu$ and every $k=1,\ldots, n$, where $kA:=A+\cdots+A$ is the $k$-fold sumset of $A$ and $\text{dom}(\mu)$ denotes the domain of definition of $\mu$;
(ii) for each $\alpha \in [0,1]$ and every non-empty finite $B\subseteq \mathbf{N}$, there is $A \subseteq \mathbf{N}$ with $A+B \in \mathrm{dom}(\mu)$ and $\mu(A+B)=\alpha$ for every quasi-density $\mu$;
(iii) for each $\alpha \in [0,1]$, there exists $A\subseteq \mathbf{N}$ with $2A = \mathbf{N}$ such that $A \in \text{dom}(\mu)$ and $\mu(A) = \alpha$ for every quasi-density $\mu$.
Proofs rely on the properties of a little known density first considered by R.\,C.~Buck and the ``structure'' of the set of all quasi-densities; in particular, they are rather different than previously known proofs of special cases of the same results. 
\end{abstract}
\maketitle
\thispagestyle{empty}

\section{Introduction}\label{sec:introduction}
Given $X_1, \ldots, X_n \subseteq \mathbf Z$, we denote by $X_1 + \cdots + X_n$ the \emph{sumset} of $X_1, \ldots, X_n$ (i.e., the set of all sums of the form $x_1 + \cdots + x_n$ with $x_i \in X_i$ for all $i = 1, \ldots, n$); in particular, we write $kX$ for the \emph{$k$-fold sumset} (i.e., the sumset of $k$ copies) of a given $X \subseteq \bf Z$. Sumsets are some of the most fundamental objects studied in additive combinatorics \cite{Na96,Ru09}, with a great variety of results relating the ``size'' of the summands $X_1, \ldots, X_n$ to that of the sumset $X_1 + \cdots + X_n$. 

When the summands are finite, the size is usually the number of elements. Otherwise, many different notions of size have been considered, each corresponding to some real-valued function (either totally or partially defined on the power set of $\bf Z$) that, while retaining essential features of a probability, is better suited than a measure to certain applications.
In the latter case, the focus has definitely been on the asymptotic density $\mathsf d$, the lower asymptotic density $\mathsf d_\star$, and the Schnirelmann density $\sigma$, where we recall that, for a set $X \subseteq \bf N$,
\[
\mathsf d(X) := \lim_{n \to \infty} \frac{|X \cap \llb 1, n \rrb|}{n},
\quad 
\mathsf d_\star(X) := \liminf_{n \to \infty}  \frac{|X \cap \llb 1, n \rrb|}{n},
\quad\text{and}\quad
\sigma(X) := \inf_{n \ge 1} \frac{|X \cap \llb 1, n \rrb|}{n},
\]
with the understanding that the limit in the definition of $\mathsf d(X)$ has to exist. It is entirely beyond the scope of this manuscript to provide a survey of the relevant literature, so we limit ourselves to list a couple of classical results that are somehow related with our work:

\begin{itemize}
\item In \cite{Vo57} (see, in particular, the last paragraph of the section ``Added in proof''), B.~Volkmann proved that, for all $n \ge 2$ and $\alpha_1, \ldots, \alpha_n, \beta \in {]0,1]}$ with $\alpha_1 + \cdots + \alpha_n \le \beta$, there are $A_1, \ldots, A_n \subseteq \mathbf N$ such that $\mathsf d(A_i) = \alpha_i$
for each $i = 1, \ldots, n$ and $\mathsf d(A_1 + \cdots + A_n) = \beta$.
\item In \cite[Theorem 1]{Na90}, M.\,B.~Nathanson showed that, for $n \ge 2$ and all $\alpha_1, \ldots, \alpha_n, \beta \in [0,1]$ with $\alpha_1 + \cdots + \alpha_n \le \beta$, there exist $X_1, \ldots, X_n \subseteq \mathbf N$ with $\mathsf d_\star(X_1) = \sigma(X_1) = \alpha_i$ for each $i = 1, \ldots, n$ and $\mathsf d_\star(X_1 + \cdots + X_n) = \sigma(X_1 + \cdots + X_n) = \beta$. 
\end{itemize}
In a similar vein, A.~Faisant et al.~have more recently proved the following (see \cite[Theorem 1.3]{Grekos2018}):

\begin{theorem}\label{thm:grekos(1)}
	Given $n \in \mathbf N^+$ and $\alpha \in [0,1]$, there is $A \subseteq \mathbf N$ with $\mathsf d(kA) = k\alpha/n$ for each $k = 1, \ldots, n$.
\end{theorem}
Their proof combines the equidistribution theorem (i.e., that the sequence $n \mapsto na \bmod 1$ is uniformly distributed in the interval $[0,1]$ for every irrational number $a$) with the elementary property that, for every $\alpha \in {]0,1]}$, the asymptotic density of the set 
$\left\{\lfloor \alpha^{-1}n \rfloor: n \in \mathbf N\right\}$
is equal to $\alpha$.
In the same man\-u\-script, one can also find the following (see \cite[Theorem 1.2]{Grekos2018}):

\begin{theorem}
\label{thm:grekos(2)}
Given $\alpha \in [0,1]$ and a non-empty finite $B \subseteq \mathbf N$, there is $A \subseteq \mathbf N$ with $\mathsf d(A+B) = \alpha$.
\end{theorem} 
This is a partial generalization of Theorem \ref{thm:grekos(1)} for the special case where $n = 1$. 
A complete generalization was obtained by P.-Y.~Bienvenu and F.~Hennecart, shortly after \cite{Grekos2018} was posted on arXiv in Sept.~2018: Their proof is based on a ``finite version'' of Weyl's criterion for equidistribution due to P.~Erd{\H o}s and P.~Tur\'an
(see \cite[Theorem 1.8]{BH2019} for details and  \cite[Theorems 1.1.a and 1.5]{BH2019} for additional results along the same lines).

Yet another item in the spirit of Theorem \ref{thm:grekos(1)} is the following result by N.~Hegyv\'ari et al.~(see \cite[Proposition 2.1]{HHP2019}):

\begin{proposition}\label{prop:he-he-pa}
	Given $\alpha \in [0,1]$, there is $A \subseteq \mathbf N$ with $0 \in A$ and $\gcd(A) = 1$ such that $\mathsf d(A) = \alpha$ and $2A = \mathbf N$.
\end{proposition} 
In the present paper, we aim to prove that Theorems \ref{thm:grekos(1)} and \ref{thm:grekos(2)} and Proposition  \ref{prop:he-he-pa} hold, much more generally, with the asymptotic density $\mathsf d$ replaced by an arbitrary \emph{quasi-density} $\mu$ (see Sect.~\ref{subsec:densities} for definitions) and --- what is perhaps more interesting --- uniformly in the choice of $\mu$ (see Theorems \ref{thm:range&sumsets(1)}--\ref{thm:bases} for a precise formulation). Most notably, this implies that Theorems \ref{thm:grekos(1)} and \ref{thm:grekos(2)} are true with $\mathsf d$ replaced by the Banach density \cite[Sect.~5.7]{Ru09} or the analytic density \cite[Sect.~III.1.3]{Te95}, both of which play a rather important role in number theory and related fields and for which we are not aware of any similar results in the literature.

We emphasize that the proofs of our 
generalizations of Theorems \ref{thm:grekos(1)} and \ref{thm:grekos(2)} take a completely different route than the ones found in \cite{BH2019,Grekos2018}: 
The latter critically depend on special features of the asymptotic density, whereas our approach relies on the properties of a little known density first considered by R.\,C.~Buck \cite{Bu0} and the ``structure'' of the set of all quasi-densities. This is in line with one of our long-term goals, which was also the motivation for first introducing quasi-densities in \cite{LT}: Obtain sharper versions of various results in additive combinatorics and analytic number theory by shedding light on the ``(minimal) structural properties'' they depend on.

\section{Preliminaries}\label{sec:preliminaries}
In this section, we establish some notations and terminology used throughout the paper and prepare the ground for the proofs of our main theorems in Sect.~\ref{sec:main-results}.

\subsection{Generalities}
\label{subsec:generalities}
We denote by $\mathbf R$ the real numbers, by $\mathbf H$ either the integers $\bf Z$ or the non-negative integers $\bf N$, and by $\mathbf N^+$ the positive integers. Given $x \in \mathbf R$, we use $\lfloor x \rfloor$ for the greatest integer $\le x$ and set $\mathrm{frac}(x) := x - \lfloor x \rfloor$; and given $X \subseteq \mathbf Z$ and $h,k \in \mathbf Z$, we define $k \cdot X + h := \{kx+h: x \in X\}$. An \emph{arithmetic progression} of $\HHb$ is then a set of the form $k\cdot \HHb+h$ with $k \in \mathbf N^+$ and $h \in \mathbf H$, and we write 
\begin{itemize}
\item $\mathscr{A}$ for the collection of all finite unions of arithmetic progressions of $\HHb$;
\item $\mathscr{A}_\infty$ for the collection of all subsets of $\mathbf H$ that can be expressed as the union of a finite set and countably many arithmetic progressions of $\bf H$;
\item $\llb a, b \rrb := \{x \in \mathbf Z: a \le x \le b\}$ for the discrete interval between two integers $a$ and $b$.
\end{itemize}
If $X$ and $Y$ are sets, then we write $\mathcal P(X)$ for the power set of $X$ and $X \subseteq_\textup{fin} Y$ to mean that $X \setminus Y$ is finite. Further terminology and notations, if not ex\-plained when first introduced, are standard, should be clear from context, or are borrowed from \cite{LT}.

\subsection{Densities (and quasi-densities)}
\label{subsec:densities}

We say a function $\mu^\star: \mathcal P(\mathbf H) \to \mathbf R$ is an \emph{upper density} (on $\mathbf H$) provided that, for all $X, Y \in \mathcal P(\mathbf H)$, the following conditions are satisfied:

\begin{enumerate}[label={\rm (\textsc{f}\arabic{*})}]
	\item\label{it:f1} $\mu^\star(X) \le \mu^\star(\mathbf H) = 1$;
	\item\label{it:f2} $\mu^\star$ is \emph{monotone}, i.e., if $X \subseteq Y$ then $\mu^\star(X) \le \mu^\star(Y)$;
	\item\label{it:f3} $\mu^\star$ is \emph{subadditive}, i.e., $\mu^\star(X \cup Y) \le \mu^\star(X) + \mu^\star(Y)$;
	\item\label{it:f4} $\mu^\star(k \cdot X + h) = \frac{1}{k} \mu^\star(X)$ for every $k \in \mathbf N^+$ and $h \in \mathbf H$.
\end{enumerate}
In addition, we say $\mu^\star$ is an  \emph{upper quasi-density} (on $\bf H$) if it satisfies \ref{it:f1}, \ref{it:f3}, and \ref{it:f4}. 

It is arguable that non-monotone upper quasi-densities --- whose existence is guaranteed by \cite[Theorem 1]{LT} --- are not so interesting from the point of view of applications. Yet, it seems meaningful to understand if monotonicity is ``critical'' to our conclusions or can be dispensed with: This is basically our motivation for considering
upper quasi-densities in spite of our main interest lying in the study of upper densities (it is obvious that every upper density is an upper quasi-density).
 
With the above in mind, we let the \emph{conjugate} of an upper quasi-density $\mu^\star$ be the function
\[
\mu_\star: \mathcal{P}(\HHb) \to \RRb: X\mapsto 1-\mu^\star(\HHb\setminus X).
\]
Then we refer to the restriction $\mu$ of $\mu^\star$ to the set 
$$
\mathcal D := 
\{X \subseteq \HHb: \mu^\star(X) = \mu_\star(X)\}
$$ 
as the \emph{quasi-density} induced by $\mu^\star$, or simply as a quasi-density (on $\bf H$) if explicit reference to $\mu^\star$ is unnecessary. Accordingly, we call $\mathcal D$ the \emph{domain} of $\mu$ and denote it by $\text{dom}(\mu)$.

Upper densities (and upper quasi-densities) were first introduced in \cite{LT} and further studied in \cite{LTsmall2019,LT17}. Notable examples include the upper asymptotic, upper Banach, upper analytic, upper logarithmic, upper P\'olya, and upper Buck densities, see \cite[Sect.~6 and 
Examples 4, 5, 6, and 8]{LT} for details. In particular, we recall that the \emph{upper Buck density} (on $\bf H$) is the function
\begin{equation}\label{equ:def-upper-buck}
\mathfrak{b}^\star: \mathcal{P}(\HHb) \to \RRb: X \mapsto \inf_{A \in \mathscr{A},\, X\subseteq A}\mathsf{d}^\star(A \cap \mathbf N),
\end{equation}
where $\mathscr A$ is the collection of all finite unions of arithmetic progressions of $\HHb$ (as already mentioned in Sect.~\ref{subsec:generalities}) and $\mathsf d^\star$ is the \emph{upper asymptotic density} on $\bf N$, that is, the function
\begin{equation}\label{eq:asymptotic-density}
\mathcal P(\mathbf N) \to \mathbf R: X \mapsto \limsup_{n \to \infty} \frac{|X \cap \llb 1, n \rrb|}{n}.
\end{equation}
We shall write $\mathfrak b_\star$ and $\mathfrak b$, respectively, for the conjugate of and the density induced by $\mathfrak b^\star$; we call $\mathfrak b_\star$ the \emph{lower Buck density} and $\mathfrak b$ the \emph{Buck density} (on $\bf H$). By \cite[Example 5]{LT}, one has
\begin{equation}\label{equ:lower-buck-density}
\mathfrak b_\star(X) = \sup_{A \in \mathscr A,\, A \subseteq X} \mathsf d^\star(A \cap \mathbf N), \quad \text{for every } X \subseteq \mathbf H.
\end{equation}
Note that the density induced by and the conjugate of $\mathsf d^\star$ are, resp., the asymptotic density $\mathsf d$ and the lower asymptotic density $\mathsf d_\star$ introduced in Sect.~\ref{sec:introduction}: One should keep this in mind when comparing our main results (that is, Theorems \ref{thm:range&sumsets(1)}--\ref{thm:bases}) with Theorems \ref{thm:grekos(1)} and \ref{thm:grekos(2)} and Proposition \ref{prop:he-he-pa}.

\subsection{Basic properties}
\label{subec:basic-properties}
Our primary goal in this section is to prove an inequality for the upper and the lower Buck density of sumsets of a certain special form (Proposition \ref{prop:sumset}). We start with a recollection of basic facts that are either implicit to or already contained in \cite{LT}.

\begin{proposition}\label{prop:basic}
	Let $\mu^\star$ be an upper quasi-density on $\HHb$. 
The following hold:
	\begin{enumerate}[label={\rm (\roman{*})}]
		\item \label{it:prop:basic(1)} $\mathfrak b_\star(X) \le \mu_\star(X)\le \mu^\star(X) \le \mathfrak b^\star(X)$ for every $X\subseteq \HHb$.
		\item \label{it:prop:basic(3)} If $h \in \mathbf H$ and $X \subseteq Y \subseteq \bf H$, then $\mathfrak b_\star(X+h) = \mathfrak b_\star(X) \le \mathfrak b_\star(Y)$.
		\item \label{it:prop:basic(4)} $\mathscr{A}\subseteq \mathrm{dom}(\mathfrak{b}) \subseteq \mathrm{dom}(\mu)$ and $\mu(X)=\mathfrak{b}(X)$ for every $X \in \mathrm{dom}(\mathfrak b)$.
		\item \label{it:prop:basic(5)} If $m \in \mathbf N^+$ and $\mathfrak h \subseteq \llb 0, m-1 \rrb$, then $m \cdot \mathbf H + \mathfrak h \in \mathrm{dom}(\mathfrak b)$ and $\mathfrak b(m \cdot \mathbf H + \mathfrak h) = \frac{|\mathfrak h|}{m}$.
		\item \label{it:prop:basic(6)} If $X\subseteq \HHb$ is finite, then $X \in \mathrm{dom}(\mathfrak b)$ and $\mathfrak b(X)=0$. 
		\item \label{it:prop:basic(7)} If $X \in \mathrm{dom}(\mathfrak b)$, $Y \subseteq \mathbf H$, and $\mathfrak b^\star(Y) = 0$, then $X \cup Y \in \mathrm{dom}(\mathfrak b)$ and $\mathfrak b(X \cup Y) = \mathfrak b(X)$.
	\end{enumerate}
\end{proposition}
\begin{proof}
We have already mentioned that $\mathfrak b^\star$,
as defined in Eq.~\eqref{equ:def-upper-buck},
is an upper density
and hence monotone. 
With this in mind,
 \ref{it:prop:basic(1)} follows from \cite[Proposition 2(vi), Theorem 3, and Corollary 4]{LT}, 
where among other things it is established that $\mathfrak b^\star$ is the pointwise maximum of the set of all upper quasi-densities on $\HHb$;
 \ref{it:prop:basic(3)} follows from  \cite[Proposition 15]{LT} (which shows that $\mathfrak b_\star$ is ``shift-invariant'') and the monotonicity of $\mathfrak b^\star$; \ref{it:prop:basic(4)} and \ref{it:prop:basic(5)} follow from \cite[Corollary 5 and Proposition 7]{LT}; and \ref{it:prop:basic(6)} follows from \ref{it:prop:basic(1)} and \cite[Proposition 6]{LT}. As for 
	\ref{it:prop:basic(7)}, note that, if $X \in \dom(\mathfrak b)$, $Y \subseteq \bf H$, and $\mathfrak b^\star(Y) = 0$, then we have from \ref{it:prop:basic(1)}, \ref{it:prop:basic(3)}, and \ref{it:f3} that 
\[
	\mathfrak b^\star(X)  = \mathfrak b_\star(X) \le \mathfrak b_\star(X \cup Y) \le \mathfrak b^\star(X \cup Y) \le \mathfrak b^\star(X) + \mathfrak b^\star(Y) = \mathfrak b^\star(X),
\]
which proves that $X \cup Y \in \dom(\mathfrak b)$ and $\mathfrak b(X \cup Y) = \mathfrak b(X)$, as wished.
\end{proof}

The next result shows that $\mathfrak b^\star$ and $\mathfrak b_\star$ are additive under some circumstances.
\begin{proposition}\label{lem:buckondisjoint}
	Let $X, Y \subseteq \bf H$ and $A, B \in \mathscr A$, and assume $X \subseteq A$, $Y \subseteq B$, and $A \cap B = \emptyset$. Then $\mathfrak b^\star(X \cup Y) = \mathfrak b^\star(X) + \mathfrak b^\star(Y)$ and $\mathfrak b_\star(X \cup Y) = \mathfrak b_\star(X) + \mathfrak b_\star(Y)$.
\end{proposition}

\begin{proof}
Given $E,F,G \in \mathscr{A}$ with $
	X\subseteq E$, $
	Y\subseteq F$, and $
	G \subseteq X \cup Y$, 
	it is clear from our assumptions that
	\begin{equation}\label{equ:(3a)}
	X \subseteq E\cap A \in \mathscr{A}, 
	\quad
	Y \subseteq F \cap B \in \mathscr{A},
	\quad\text{and}\quad
	(E \cap A) \cap (F \cap B) \subseteq A \cap B = \emptyset,
	\end{equation}
	and
	\begin{equation}\label{equ:(4a)}
	\left\{
	\begin{array}{l}
	\mathscr A \ni G \cap A \subseteq X
	\quad\text{and}\quad
	\mathscr A \ni G \cap B \subseteq Y, \\
	G = (G \cap A) \cup (G \cap  B)
	\quad\text{and}\quad
	(G \cap A) \cap (G \cap B) = \emptyset.
	\end{array}
	\right.
	\end{equation}
	On the other hand, we have by parts \ref{it:prop:basic(4)} and \ref{it:prop:basic(5)} of Proposition \ref{prop:basic} that
	\[
	\mathsf d^\star(V \cup W) = \mathsf d^\star(V) + \mathsf d^\star(W),
	\quad\text{for all } V, W \in \mathscr A \text{ with }V \cap W = \emptyset;
	\]
	and it is a basic fact that, for all non-empty subsets $S$ and $T$ of $\mathbf R$,
	\[
	\inf S + \inf T = \inf (S + T)
	\quad\text{and}\quad
	\sup S + \sup T = \sup(S + T).
	\] 
	So, putting it all together, we conclude from Eqs.~\eqref{equ:def-upper-buck} and \eqref{equ:(3a)} that
	\[
	\begin{split}
	\mathfrak{b}^\star(X)+\mathfrak{b}^\star(Y) & = \inf\{\mathsf d^\star(E) + \mathsf d^\star(F): E, F \in \mathscr A,\, X \subseteq E, \text{ and } Y \subseteq F\} \\
	& \le \inf\{\mathsf d^\star(E \cap A) + \mathsf d^\star(F \cap B): E, F \in \mathscr A,\, X \subseteq E, \text{ and } Y \subseteq F\} \\
		& \le \inf\{\mathsf d^\star((E \cup F) \cap (A \cup B)): E, F \in \mathscr A,\, X \subseteq E, \text{ and } Y \subseteq F\} \\
	& \leq \inf\{\mathsf d^\star(G): G \in \mathscr A \text{ and }X \cup Y \subseteq G\} \\
	& = \mathfrak b^\star(X \cup Y),
	\end{split} 
	\]
	which, by subadditivity of $\mathfrak b^\star$, leads to $\mathfrak{b}^\star(X \cup Y) = \mathfrak{b}^\star(X)+\mathfrak{b}^\star(Y)$. Likewise, Eqs.~\eqref{equ:lower-buck-density} and \eqref{equ:(4a)} yield
	\[
	\begin{split}
	\mathfrak b_\star(X \cup Y) 
		& = \sup\{\mathsf d^\star(G): G \in \mathscr A \text{ and } G \subseteq X \cup Y\} \\
		& = \sup \{\mathsf d^\star(E \cup F): E, F \in \mathscr A, \, E \subseteq X,\text{ and }F \subseteq Y\} \\
		& = \sup\{\mathsf d^\star(E) + \mathsf d^\star(F): E, F \in \mathscr A, \, E \subseteq X,\text{ and }F \subseteq Y\} \\
		& = \mathfrak b_\star(X) + \mathfrak b_\star(Y);
	\end{split}
	\]
in particular, it is seen from Eq.~\eqref{equ:(4a)} that, if $G \in \mathscr A$ and $G \subseteq X \cup Y$, then $\mathscr A \ni G \cap A \subseteq X$ and $\mathscr A \ni G \cap B \subseteq Y$, which is used in the second equality from the last block.
\end{proof}
It is perhaps worth noticing that Proposition \ref{lem:buckondisjoint} does not hold with $\mathfrak b^\star$ replaced by $\mathsf d^\star$. In fact, set $X := E \cap (2 \cdot \mathbf H)$ and $Y := F \cap (2 \cdot \mathbf H + 1)$, where 
\[
E := \bigcup_{n \ge 1} \llb (4n)!, (4n+1)! \rrb
\quad\text{and}\quad
F := \bigcup_{n \ge 1} \llb (4n+2)!, (4n+3)! \rrb.
\]
Then $X$ and $Y$ are both contained in disjoint arithmetic progressions of $\bf H$, but it is not difficult to see that $\mathsf d^\star(X) = \mathsf d^\star(Y) = \mathsf d^\star(X \cup Y) = \frac{1}{2}$, cf. \cite[Lemma 1]{LT}.

\begin{lemma}\label{lem:extra-fix}
Given $k, p, q \in \mathbf N^+$, $\mathfrak a \subseteq \llb 0, q-1 \rrb$, and $r \in \bf H$ with $\gcd(p,q) = 1$, let $A := q \cdot \mathbf H + \mathfrak a$ and $B := p \cdot \mathbf H + r$. Then the sets $k(A \cap B)$ and $kA \cap kB$ are both in $\mathscr A$ and their symmetric difference is finite; in particular, $k(A \cap B), kA \cap kB \in \dom(\mathfrak b)$. Moreover, $\mathfrak b(k(A \cap B)) = \mathfrak b(kA \cap kB) = (pq)^{-1} |k \mathfrak{a}|$.
\end{lemma}

\begin{proof}
We can assume $\mathfrak a \ne \emptyset$, or else the conclusion is trivial. It is also clear that, if $X = m \cdot \mathbf H + \mathfrak q$ for some $m \in \mathbf N^+$ and finite $\mathfrak q \subseteq \bf H$, then $kX = m \cdot \mathbf H + k \mathfrak q \in \mathscr A$; and it is obvious that $k(A \cap B) \subseteq kA \cap kB$ (because $x_1 + \cdots + x_k \in kA \cap kB$ for all $x_1, \ldots, x_k \in A \cap B$). Since $A \cap B \in \mathscr A$ and, by Proposition \ref{prop:basic}\ref{it:prop:basic(5)}, $\mathscr A \subseteq \dom(\mathfrak b)$, we are thus left to check that 
\[
\text{(i) } kA \cap kB \subseteq_\textup{fin} k(A \cap B)
\quad\text{and}\quad \text{(ii) } \mathfrak b(kA \cap kB) = (pq)^{-1} | k \mathfrak a|.
\]
(i) Pick $x \in kA \cap kB$ and, in case $\HHb=\NNb$, assume $x \ge k(k-1)pq$. Then $x \equiv kr \bmod p$ and there exist $a_1, \ldots, a_k \in A$ with $a_1 \le \cdots \le a_k$ such that $x = a_1 + \cdots + a_k$ (observe that, if $\mathbf H = \bf N$, then $a_k \ge (k-1)pq$). Since $p$ and $q$ are coprime, we gather from the Chinese remainder theorem that, for each $i \in \llb 1,k-1\rrb$, there is a smallest integer $a_i^\prime \ge a_i$ such that $a_i' \equiv a_i \bmod q$ and $a_i' \equiv r \bmod p$ (in particular, $a_i' \le a_i + pq$). Accordingly, set $a_k^\prime := x-\sum_{i=1}^{k-1} a_i^\prime$. 
By construction, $a_k' \equiv x - \sum_{i=1}^{k-1} a_i^\prime \equiv a_k \bmod q$ and $a_k^\prime \equiv kr - (k-1)r \equiv r \bmod{p}$. Moreover, if $\HHb=\NNb$, then 
$$
a_k^\prime = a_k-\sum_{i=1}^{k-1}(a_i^\prime-a_i) \ge (k-1)pq-(k-1)pq \ge 0.
$$
In consequence, we find that $a_1^\prime, \ldots, a_k^\prime \in A \cap B$ and hence $x = a_1^\prime+\cdots+a_k^\prime \in k(A \cap B)$. This suffices to complete the proof, because $x$ is an arbitrary element of $(kA \cap kB) \setminus V$, where $V := \llb 0, k(k-1)pq-1 \rrb$ if $\mathbf H = \bf N$ and $V := \emptyset$ otherwise (to wit, $V$ is a finite set).

\vskip 0.05cm

(ii) We have $kA = q \cdot \mathbf H + k\mathfrak a$ and hence $kA \cap kB = (q \cdot \mathbf H + k\mathfrak a) \cap (p \cdot \mathbf H + kr)$. Since $\gcd(p,q) = 1$, it follows from the Chinese remainder theorem that $kA \cap kB$ is, apart from finitely many elements, the union of $|k\mathfrak a|$ pairwise disjoint arithmetic progressions modulo $pq$. Therefore, we conclude from parts \ref{it:prop:basic(5)}--\ref{it:prop:basic(7)} of Proposition \ref{prop:basic} that $\mathfrak b(kA \cap kB) = (pq)^{-1}|k \mathfrak a|$, as wished.
\end{proof}

\begin{proposition}\label{prop:sumset}
Fix $n,t, p, q \in \NNb^+$ and $s \in \mathbf N$ such that $nt<q$ and $\gcd(p,q) = 1$, let $Y$ be a non-empty subset of $q \cdot \mathbf H + t$, and define
$
X := q \cdot \HHb + \llb 0, t-1 \rrb$, $V := p \cdot \mathbf H + s$, and 
$S := (X \cup Y) \cap V$.
Then 
\begin{equation}\label{equ:(2)-19-09-2021}
\frac{kt}{pq} + \mathfrak b_\star(k(Y \cap V))
= \mathfrak b_\star(kS) \le \mathfrak b^\star(kS) = \frac{kt}{pq} + \mathfrak b^\star(k(Y \cap V)) \le \frac{kt+1}{pq},
\quad\text{for every } k \in \llb 1, n \rrb.
\end{equation}
In particular, if $Y \in \mathscr A$, then 
\[
kS, k(Y  \cap V) \in \dom(\mathfrak b)
\quad\text{and}\quad
\mathfrak b(kS) = \frac{kt}{pq} + \mathfrak b(k (Y \cap V)),
\quad\text{for every } k \in \llb 1, n \rrb.
\]
\end{proposition}

\begin{proof}
	The ``In particular'' part of the statement is straightforward from Eq.~\eqref{equ:(2)-19-09-2021} and Proposition \ref{prop:basic}\ref{it:prop:basic(6)}, by the fact that $mA \in \mathscr A$ for all $m \in \mathbf N^+$ and $A \in \mathscr A$. So, we focus on the rest.
	
	Fix $k \in \llb 1, n \rrb$, and define $X' := X \cap V \in \mathscr A$, $Y' := Y \cap V$, and $V' := (q \cdot \mathbf H + t) \cap V \in \mathscr A$. 
	Since $Y$ is a non-empty subset of $q \cdot \mathbf H + t$ and $p$ is coprime to $q$ (by hypothesis), we gather from the Chinese re\-main\-der theorem that $pqx + r \in Y' \subseteq V' = pq \cdot \bf H + r$ for some $x \in \mathbf H$ and $r \in q \cdot \mathbf H + t$. Using that $X'$ is itself a finite union of arithmetic progressions modulo $pq$, it follows that, for all $i \in \mathbf N^+$ and $j \in \mathbf N$, 
	\begin{equation}\label{equ:(1)-18-09-2021}
	iX' + jV' = iX' + jr \subseteq_\textup{fin} iX' + j(pqx + r) \subseteq iX' + jY' \subseteq iX' + jV' \in \mathscr A
	\end{equation}
	and, on the other hand,
	\begin{equation}\label{equ:(2)-18-09-2021}
	iX' + jY' \subseteq iX + j(q \cdot \mathbf H + t) = q \cdot \mathbf H + \llb 0, it-i \rrb + jt = q \cdot \mathbf H + \llb jt, (i+j)t-i \rrb;
	\end{equation}
	in particular, the relation $iX' + jr \subseteq_\textup{fin} iX' + j(pqx + r)$ becomes an equality when $\mathbf H = \mathbf Z$. 
	Hence,
	\begin{equation}\label{equ:(3)-18-09-2021}
	kY' \subseteq kV' \subseteq q \cdot \mathbf H + kt \in \mathscr A
	\end{equation}
	and
	\begin{equation}\label{equ:(1)-19-09-2021}
	Z_k := \bigcup_{i=1}^{k} (iX'+(k-i)Y') \subseteq \bigcup_{i=1}^k (iX'+(k-i)V') =: Z_k' \in \mathscr A.
	\end{equation}
	So taking into account that
	\[
	kS = \bigcup_{i=0}^k (iX' + (k-i)Y') = Z_k \cup kY'
	\]
	and considering that $(k - i)t \le kt - (i + 1) + 1 \le nt < q$ for all $i \in \mathbf N$ (by hypothesis) and, by Eq.~\eqref{equ:(2)-18-09-2021},
	\[
	Z_k' \subseteq \bigcup_{i=1}^{k} (q \cdot \mathbf H + \llb (k-i)t, kt - i \rrb) = q \cdot \mathbf H + \llb 0, kt - 1 \rrb,
	\]
	we gather from Propositions \ref{prop:basic}\ref{it:prop:basic(3)} and \ref{lem:buckondisjoint} and Eq.~\eqref{equ:(3)-18-09-2021} that
	\[
	\mathfrak b_\star(Z_k) + \mathfrak b_\star(kY') = \mathfrak b_\star(kS) \le \mathfrak b^\star(kS) = \mathfrak b^\star(Z_k) + \mathfrak b^\star(kY') \leq \mathfrak b^\star(Z_k) + \mathfrak b^\star(kV') = \mathfrak b^\star(Z_k) + \frac{1}{pq}.
	\]
	It remains to see that $\mathfrak b_\star(Z_k) = 	\mathfrak b^\star(Z_k) = (pq)^{-1} kt$. For, set
	\begin{equation}\label{eq:(4)-18-09-2021}
	S' := X' \cup V' = (X \cup (q \cdot \mathbf H + t)) \cap V = (q \cdot \mathbf H + \llb 0, t \rrb) \cap V.
	\end{equation}
	We have
	\[
	kS' = \bigcup_{i=0}^k (iX' + (k-i)V') = Z_k' \cup kV' \in \mathscr A.
	\]
	Recalling that each of $kS'$, $Z_k'$, and $kV'$ is a finite union of arithmetic progressions (and hence, by Prop\-o\-si\-tion \ref{prop:basic}\ref{it:prop:basic(5)}, a set in the domain of $\mathfrak b$) with $kV' \subseteq q \cdot \mathbf H + kt$ (see Eq.~\eqref{equ:(3)-18-09-2021}) and $kt < q$, it thus follows from Eq.~\eqref{eq:(4)-18-09-2021}, Lemma \ref{lem:extra-fix}, and Propositions \ref{prop:basic}\ref{it:prop:basic(6)} and \ref{lem:buckondisjoint} that
	\begin{equation*}\label{eq:(5)-18-09-2021}
	\frac{kt+1}{pq} = \mathfrak b(kS') = \mathfrak b(Z_k') + \mathfrak b(kV') = \mathfrak b(Z_k') + \frac{1}{pq}.
	\end{equation*}
    Moreover, we have from Eqs.~\eqref{equ:(1)-18-09-2021} and \eqref{equ:(1)-19-09-2021} that $Z_k' \subseteq_\textup{fin} Z_k \subseteq Z_k'$. Therefore, we conclude from the last display and Proposition \ref{prop:basic}\ref{it:prop:basic(7)} that $Z_k \in \dom(\mathfrak b)$ and $\mathfrak b(Z_k) = \mathfrak b(Z_k') = (pq)^{-1} kt$ (as wished).
\end{proof}

\subsection{A positional representation}
\label{subsec:positional-reps}
We introduce a non-standard positional representation of real numbers (Proposition \ref{prop:expansion}) that will be of key importance in the proof of Theorem \ref{thm:range&sumsets(1)}; cf. \cite[Theorem 1.6]{Ni06} for an ``analogous'' result attributed by I.~Niven to G.~Cantor.

\begin{lemma}\label{lem:divisibility}
	Let $\alpha$ be an irrational number in the interval $[0,1]$, and fix $m,t\in \NNb^+$. There then exist infinitely many $n \in \NNb^+$ such that
	$\lfloor (nt+1)\alpha \rfloor \in m \cdot \mathbf N^+$.
\end{lemma}
\begin{proof}
	Since $t\alpha$ is irrational, the sequence $(\mathrm{frac}(Nt\alpha))_{N\ge 0}$ is equi\-distributed in $[0,1[$\,. 
	This implies that there exists a set $\mathcal{N}\subseteq \NNb^+$ such that $\mathsf{d}(\mathcal{N})=(1-\alpha)/m$ and $\mathrm{frac}(Nt\alpha) \in \left]0,(1-\alpha)/m\right[$ for all $N \in \mathcal{N}$, see e.g. \cite[Exercise 1.15, p. 6]{MR0419394}. Since
	$$
	\mathrm{frac}((Ntm+1)\alpha)=m\mathrm{frac}(Nt\alpha)+\alpha \in  {]0,1[}\,,
	$$
	it follows that $\lfloor (Ntm+1)\alpha\rfloor=m\lfloor Nt\alpha \rfloor \in m \cdot \mathbf N^+$ for all $N \in \mathcal{N}$.
\end{proof}

\begin{proposition}\label{prop:expansion}
Let $\alpha$ be an irrational number in the interval $[0,1]$, and fix $n \in \mathbf N^+$. There then exist sequences $(\beta_i)_{i \ge 1}$ and $(q_i)_{i \ge 0}$ of positive integers with $q_0 = 1$ such that 
\begin{equation}\label{equ:cond1}
\alpha = \sum_{i \ge 1} \frac{n!\,\beta_i}{q_1 \cdots q_i}
\end{equation}
and, for every $i \in \mathbf N^+$, 
\begin{equation*}
\gcd(q_i,n q_0 \cdots q_{i-1}) = 1,
\quad 
\alpha_{i-1} \in {]0,1[}\,, 
\quad\text{and}\quad
\lfloor q_i \alpha_{i-1} \rfloor \in n! \cdot \mathbf N^+,
\end{equation*}
where we have defined
\begin{equation}\label{eq:definitionalphai}
\alpha_0 := \alpha
\quad\text{and}\quad 
\alpha_i := q_1 \cdots q_{i} \left(\alpha - \sum_{j=1}^i \frac{n!\,\beta_j}{q_1 \cdots q_j}\right).
\end{equation}
\end{proposition}

\begin{proof}
Given $x \in [0,1]$ and $N \in \mathbf N^+$, let 
$$
\mathcal{Q}(x,N) := \left\{q \in \NNb^+ \colon \mathrm{gcd}(q,N)=1 \text{ and }\lfloor qx \rfloor \in n! \cdot \mathbf N^+\right\};
$$
it follows by Lemma \ref{lem:divisibility} that, if $x$ is irrational, then the set $\mathcal{Q}(x,N)$ is infinite.
Thus, since $\alpha_0, \alpha_1, \ldots$ are all irrational numbers by their definition in Eq.~\eqref{eq:definitionalphai} and the irrationality of $\alpha$, we can recursively define sequences $(q_i)_{i\ge 0}$ and  $(\beta_i)_{i\ge 1}$ of \textit{positive} integers by taking $q_0 := 1$ and, for each $i \in \mathbf N^+$,
\begin{equation}\label{eq:definitionqi}
q_i := \min\mathcal{Q}(\alpha_{i-1},nq_0\cdots q_{i-1})
\quad\text{and}\quad
\beta_i:=\left\lfloor \frac{q_i \alpha_{i-1}}{n!}\right\rfloor;
\end{equation} 
in particular, $\beta_i$ is a positive integer because $\lfloor q_i \alpha_{i-1} \rfloor = n!\, k_i$ for some $k_i \in \mathbf N^+$ (by definition of the set $\mathcal{Q}(\alpha_{i-1},nq_0\cdots q_{i-1})$), so that $k_i \le q_i \alpha_{i-1}/n! < k_i + 1/n!$ and hence $\beta_i = k_i$. It is clear that
\begin{equation}\label{eq:propertybetai}
q_i \alpha_{i-1}-1 < n!\, \beta_i < q_i \alpha_{i-1},
\quad \text{for every } i \in \mathbf N^+.
\end{equation}
On the other hand, $\alpha_0 =\alpha \in {]0,1[}$\,; and if $\alpha_{i-1} \in {]0,1[}$ for some $i \in \NNb^+$, then it follows by Eqs.~\eqref{eq:definitionalphai} and \eqref{eq:propertybetai} that $\alpha_{i}
=q_{i}\alpha_{i-1}-n!\,\beta_{i} \in {]0,1[}\,$. Thus, we see by induction that
\[
\alpha_i \in {]0,1[}\,, \quad\text{for all }i \in \NNb.
\]
We may note, thanks to Eq.~\eqref{eq:definitionqi}, that $q_i > q_i \alpha_{i-1} > n! \ge 1$, hence $q_i\ge 2$ for all $i \in \NNb^+$. 
To conclude, identity \eqref{equ:cond1} is obtained by the fact that
\[
\left|\,\alpha - \sum_{j=1}^i \frac{n!\,\beta_j}{q_1 \cdots q_j}\,\right|=\frac{\alpha_i}{q_1\cdots q_i} < \frac{1}{2^i}, \quad \text{for all }i \in \mathbf N^+.
\qedhere
\]
\end{proof}

\section{Main results}
\label{sec:main-results}
This section is devoted to the main results of the paper. We start with a generalization of Theorem \ref{thm:grekos(1)}. Recall from Sect.~\ref{subsec:generalities} that $\mathscr A_\infty$ denotes the family of all subsets of $\mathbf H$ that can be expressed as the union of a finite set and countably many arithmetic progressions of $\bf H$.

\begin{theorem}\label{thm:range&sumsets(1)}
	Given $n \in \mathbf N^+$ and $\alpha \in [0,1]$, there exists $A \in \mathscr A_\infty$ such that 
	$
	kA \in \mathrm{dom}(\mu)$ and $\mu(kA)=k\alpha/n$	
	for each $k \in \llb 1, n \rrb$ and every quasi-density $\mu$ on $\bf H$.
\end{theorem}

\begin{proof}
Thanks to Proposition \ref{prop:basic}\ref{it:prop:basic(4)}, it will be enough to prove that there exists $A \in \mathscr A_\infty$ such that $kA \in \text{dom}(\mathfrak b)$ and $\mathfrak b(kA) = \alpha k/n$ for each $k \in \llb 1, n \rrb$. To this end, we distinguish two cases.

\textsc{Case 1:} $\alpha$ is rational. Write $\alpha = a/b$, where $a \in \bf N$ and $b \in \mathbf N^+$. Then set 
\[
A := \{0\} \cup (nb \cdot \mathbf H + \llb 1, a \rrb)  \in \mathscr A_\infty. 
\]
Since $0 \le a \le b$, it is immediate (by induction) that  
\[
kA = \{0\} \cup (nb \cdot \mathbf H + \llb 1, ka \rrb), \quad \text{for every }k \in \llb 1, n \rrb.
\]
So, by Proposition \ref{prop:basic}\ref{it:prop:basic(4)}--\ref{it:prop:basic(7)}, we find that
\[
kA \in \text{dom}(\mathfrak b)
\quad\text{and}\quad
\mathfrak b(kA) = \frac{ka}{nb} = \frac{\alpha k}{n},
\quad\text{for every } k \in \llb 1, n \rrb.
\]

\textsc{Case 2:} $\alpha$ is irrational. By Proposition \ref{prop:expansion}, there exist sequences $(\beta_i)_{i\ge 1}$ and $(q_i)_{i\ge 0}$ of positive integers with $q_0 = 1$ such that $\gcd(q_i,n q_0 \cdots q_{i-1}) = 1$ for every $i \in \mathbf N^+$ and
\begin{equation}
\label{equ:series-expansion}
\alpha = \sum_{i \ge 1} \frac{n!\,\beta_i}{q_1 \cdots q_i}.
\end{equation}
Accordingly, we can recursively define sequences $(X_i)_{i\ge 1}$ and $(Y_i)_{i\ge 0}$ of subsets of $\HHb$ by taking $Y_0 := \HHb$ and, for each $i \in \NNb^+$,
\begin{equation}\label{eq:definitionXi}
X_{i}:=Y_{i-1} \cap (q_{i}\cdot \HHb + \llb 0, (n-1)!\,\beta_i-1 \rrb)
\quad \text{ and }\quad 
Y_{i}:=Y_{i-1} \cap (q_{i}\cdot \HHb+(n-1)!\,\beta_i).
\end{equation}
Because $q_1, q_2, \ldots$ are pairwise coprime integers, it is immediate from Eq.~\eqref{eq:definitionXi} and the Chinese remainder theorem that, for every $i \in \mathbf N^+$, there exists $r_i \in \mathbf N$ such that
\begin{equation}\label{equ:explicit-Yi}
Y_i = \bigcap_{j=1}^i (q_j \cdot \mathbf H + (n-1)! \beta_j) = q_1 \cdots q_i \cdot \mathbf H + r_i.
\end{equation}
Consequently, we obtain from Proposition \ref{prop:basic}\ref{it:prop:basic(5)} that

\begin{equation}\label{eq:bfrakyi}
kY_i \in \textrm{dom}(\mathfrak b)
\quad\text{and}\quad
\mathfrak{b}(kY_i)=\frac{1}{q_0\cdots q_{i}}\le \frac{1}{2^i}, \quad\text{for all }i, k \in \mathbf N^+.
\end{equation}
Note that the sets $X_1, X_2, \ldots$ are pairwise disjoint; moreover, 
\begin{equation}\label{equ:containment-3.1}
X_i, Y_i \in \mathscr A \setminus \{\emptyset\}
\quad\text{and}\quad
X_{i} \cup Y_{i} \subseteq Y_{i-1},
\quad \text{for every } i \in \mathbf N^+.
\end{equation}
Then, for each $i \in \NNb^+$, define $A_i := X_1 \cup \cdots \cup X_i$ and $B_i := A_i \cup Y_i$.
We set
\[
A:=\bigcup_{i\ge 1} A_i = \bigcup_{i \ge 1} X_i.
\]
It is obvious from Eq.~\eqref{equ:containment-3.1} and our definitions that $A \in \mathscr A_\infty$. So, to finish the proof, it only remains to show that $kA \in \text{dom}(\mathfrak b)$ and $\mathfrak{b}(kA)=k\alpha/n$ for all $k \in \llb 1,n\rrb$. 

For, fix $k \in \llb 1, n \rrb$ and $i \in \mathbf N^+$. 
Since $\mathfrak{b}$ is monotone, it is clear from Eqs.~\eqref{eq:bfrakyi} and \eqref{equ:containment-3.1} that 
\begin{equation}\label{equ:upper-bound-on-the-error}
\mathfrak{b}(kX_i) \le \mathfrak{b}(k(X_i\cup Y_i)) \le \frac{1}{2^{i-1}}.
\end{equation} 
On the other hand, it follows from Eq.~\eqref{equ:containment-3.1} and the above that 
\[
A_i \subseteq A \subseteq B_i
\quad\text{and}\quad
A_i, B_i \in \mathscr A \setminus \{\emptyset\},
\] 
which in turn implies that
\begin{equation}\label{eq:inequalitybstar}
kA_i \subseteq kA \subseteq kB_i,
\quad
kA_i, kB_i \in \text{dom}(\mathfrak b),
\quad\text{and}\quad
\mathfrak{b}(kA_i) \le \mathfrak{b}_\star(kA) \le \mathfrak{b}^\star(kA) \le \mathfrak{b}(kB_i).
\end{equation}
We claim that
\begin{equation}\label{equ:claim(1)}
\mathfrak{b}(kA_i)=\frac{k}{n}\sum_{j=1}^{i-1}\frac{n!\,\beta_j}{q_1\cdots q_j}+\mathfrak{b}(kX_{i}).
\end{equation}
For, let $j \in \llb 0, i-1 \rrb$ and define $Z_{i,j} := A_i \setminus A_j = X_{j+1} \cup \cdots \cup X_i$. We have from Eqs.~\eqref{eq:definitionXi} and \eqref{equ:explicit-Yi} that $X_{j+1} \subseteq Y_j$ and $Z_{i,j+1} \subseteq Z_{i,j} \subseteq Y_j$. In consequence, we see that 
\[
Z_{i,j} = X_{j+1} \cup Z_{i,j+1} = (X_{j+1} \cap Y_j) \cup (Z_{i,j+1} \cap Y_j) = (X_{j+1} \cup Z_{i,j+1}) \cap Y_j.
\]
Since each of $X_{j+1}$, $Z_{i,j+1}$, and $Y_j$ is a non-empty element of $\mathscr A$, it thus follows from Proposition \ref{prop:sumset} (applied with $q = q_{j+1}$, $t = (n-1)!\,\beta_{j+1}$, $p = q_1\cdots q_j$, $X = X_{j+1}$, $Y = Z_{i,j+1}$, and $V = Y_j$) that
\[
\mathfrak{b}(kZ_{i,j}) = \frac{k}{n}\cdot \frac{n!\,\beta_{j+1}}{q_1\cdots q_{j+1}}+\mathfrak{b}(kZ_{i,j+1}).
\]
(Note that $Z_{i,j+1} \cap Y_j = Z_{i,j+1}$.)
So considering that $A_i=Z_{i,0}$, we obtain by induction that
\[
\mathfrak{b}(kA_i)=\frac{k}{n}\cdot \frac{n! \beta_1}{q_1}+\mathfrak{b}(kZ_{i,1})=\cdots=\frac{k}{n}\sum_{j=1}^{i-1}\frac{n!\,\beta_j}{q_1\cdots q_j}+\mathfrak{b}(kZ_{i,i-1}).
\]
This suffices to prove the claim (because $X_i = Z_{i,i-1}$), and
in a similar way we find that
\begin{equation}\label{equ:claim(2)}
\mathfrak{b}(kB_i)=\frac{k}{n}\sum_{j=1}^{i-1}\frac{n!\,\beta_j}{q_1\cdots q_j}+\mathfrak{b}(k(X_{i}\cup Y_i)).
\end{equation}
The proof is essentially the same as the proof of Eq.~\eqref{equ:claim(1)}, with the sets $A_i \setminus A_j$ replaced by $B_i \setminus A_j$ ($0 \le j < i$); we omit further details.
Therefore, we gather from Eqs.~\eqref{equ:series-expansion},  \eqref{equ:upper-bound-on-the-error}, \eqref{equ:claim(1)}, and \eqref{equ:claim(2)} that
$$
\max\left\{\left|\,\mathfrak{b}(kA_i)-\frac{k\alpha}{n}\,\right|, \left|\,\mathfrak{b}(kB_i)-\frac{k\alpha}{n}\,\right|\right\}
\le 
\sum_{j\ge i}\frac{n!\,\beta_j}{q_1\cdots q_j}
+\frac{1}{2^{i-1}}.
$$
Consequently, we see that
\begin{equation*}\label{eq:limitAiandBi}
\lim_{i\to \infty}\mathfrak{b}(kA_i) = \lim_{i\to \infty}\mathfrak{b}(kB_i)=\frac{k\alpha}{n},
\end{equation*}
and we conclude, by Eq.~\eqref{eq:inequalitybstar}, that $kA \in \mathrm{dom}(\mathfrak{b})$ and $\mathfrak{b}(kA)=k\alpha/n$ (as wished).
\end{proof}

\begin{theorem}\label{thm:georges}
	Given $\alpha \in [0,1]$ and a non-empty finite set $B\subseteq \HHb$, there exists $A \in \mathscr A_\infty$ such that $A+B \in \mathrm{dom}(\mu)$  and $\mu(A+B) = \alpha$ for every quasi-density $\mu$ on $\bf H$. 
\end{theorem}
\begin{proof}
	Similarly as in the proof of Theorem \ref{thm:range&sumsets(1)}, it suffices to prove that there exists $A \in \mathscr A_\infty$ such that $A+B \in \mathrm{dom}(\mathfrak b)$  and $\mathfrak b(A+B) = \alpha$. 
To this end, set $x:=\min B$ and $y:=\max B$. 
	
	We may assume without loss of generality that $x = 0$, because $
	A+B = (A+x)+(B-x)$ and both $A+x$ and $B-x$ are subsets of $\bf H$, with $|B-x| = |B|$. Therefore, $B$ is a subset of $\mathbf N$; and we can suppose that $y \ne 0$, or else the conclusion follows by Theorem \ref{thm:range&sumsets(1)}. 
	
 	Now, the statement to be proved is trivial for $\alpha = 0$ or $\alpha = 1$ (just take $A := \emptyset$ in the former case and $A := \mathbf H$ in the latter). Consequently, let $\alpha \in {]0,1[}$ and pick $h,k \in \NNb^+$ such that 
	$$
	\frac{h}{k}<\alpha<\frac{h+1}{k}
	\quad\text{and}\quad 
	h\ge 2y+1.
	$$ 
	Then $k\alpha-h \in {]0,1[}$ and $h-y-1\ge y$, and we derive from Theorem \ref{thm:range&sumsets(1)} that there exists a set $C\in \mathscr A_\infty \cap \mathrm{dom}(\mathfrak{b})$ such that $\mathfrak{b}(C) = k\alpha - h$. So, we define
	$$
	A:=(k\cdot \HHb+\llb 0,h-y-1\rrb) \cup (k\cdot C+h-y).
	$$
	Then it is straightforward that
	$$
	A \in \mathscr A_\infty 
	\quad\text{and}\quad
	A+B=(k\cdot \HHb+\llb 0,h-1\rrb) \cup (k\cdot C+h),
	$$
	and it follows by Propositions \ref{prop:basic}\ref{it:prop:basic(5)} and \ref{lem:buckondisjoint} that 
	\[
	\mathfrak{b}^\star(A+B)=\mathfrak{b}^\star(k\cdot \HHb+\llb 0,h-1\rrb)+\mathfrak{b}^\star(k\cdot C+h)=\frac{h+\mathfrak{b}(C)}{k}=\alpha.
	\]
	Likewise, we calculate that $
	\mathfrak{b}_\star(A+B) = \alpha$. Thus, $A+B \in \text{dom}(\mathfrak b)$ and $\mathfrak b(A+B) = \alpha$.
\end{proof}

\begin{theorem}\label{thm:bases}
	Given $\alpha \in [0,1]$, there exists a set $A \subseteq \bf H$ with $0 \in A$ and $\gcd(A) = 1$ such that $2A = \mathbf H$, $A \in \mathrm{dom}(\mu)$, and $\mu(A) = \alpha$ for every quasi-density $\mu$ on $\bf H$. 
\end{theorem}

\begin{proof}
Once again, it suffices to prove that there exists $A \in \mathrm{dom}(\mathfrak b)$ such that $\mathfrak b(A) = \alpha$, cf. the proofs of Theorems \ref{thm:range&sumsets(1)} and \ref{thm:georges}. To this end, set
\[
Q:=\{x^2+y^2: x,y \in \NNb\}
\quad\text{and}\quad
X := (Q \cup (-Q)) \cap \HHb.
\]
We know from Lagrange's four square theorem that $2Q=\NNb$, and from \cite[Theorem 4.2]{LTsmall2019} that $\mathfrak{b}(Q)=0$. It follows that $2X = \mathbf H$. Moreover, it is clear from the definition of $\mathfrak b^\star$ that
\[
\mathfrak b^\star((-Q) \cap \mathbf H) = \mathfrak b^\star(Q \cap (-\mathbf H)) \le \mathfrak b^\star(Q) = 0. 
\]
Therefore, we find that
\[
X \in \mathrm{dom}(\mathfrak b)
\quad\text{and}\quad
\mathfrak{b}(X) = 0.
\]
On the other hand, Theorem \ref{thm:range&sumsets(1)} guarantees that $\mathfrak{b}(Y)=\alpha$ for some $Y \in \mathrm{dom}(\mathfrak b)$. So, letting $A:=X \cup Y$ and putting all pieces together, we get from Proposition \ref{prop:basic}\ref{it:prop:basic(7)} that 
\[
2A=\HHb, 
\quad
A \in \mathrm{dom}(\mathfrak b),
\quad\text{and}\quad
\mathfrak b(A) = \alpha. 
\]
This finishes the proof, when considering that $0 \in Q \subseteq A$ and $1 \le \gcd(A) \le \gcd(Q) = 1$.
\end{proof}

\section{Closing remarks}\label{sect:closing-remarks}

Looking at the statement of Theorem \ref{thm:range&sumsets(1)}, it is natural to ask whether assuming $A \in \mathrm{dom}(\mu)$, for some fixed quasi-density $\mu$ on $\bf H$, is sufficient to guarantee that $2A \in \mathrm{dom}(\mu)$. 

By \cite[Proposition 2.2]{HHP2019}, the answer is negative for the asymptotic density $\sf d$ on $\bf N$. But it follows by \cite[Remark 3]{LT} that, in the classical framework of Zermelo-Fraenkel set theory with the axiom of choice, there is a density $\mu$ on $\HHb$ such that $\mathrm{dom}(\mu)=\HHb$; hence, in this case, the answer is positive. 

One can still wonder what happens with the Buck density $\mathfrak b$, especially in light of the role played by $\mathfrak b$ in the proofs of Sect.~\ref{sec:main-results}.
Again, the answer turns out to be in the negative. In fact, set
$$
V:=\{n! + n: n \in \NNb\}
\quad\text{and}\quad
A:=\{x^2+y^2: x,y \in V\}.
$$
Since $\mathfrak{b}^\star$ is monotone, we gather from \cite[Theorem 4.2]{LTsmall2019}, similarly as in the proof of Theorem \ref{thm:bases}, that $A \in \textrm{dom}(\mathfrak b)$ and $\mathfrak{b}(A)=0$.
However, we will show that $2A\notin \mathrm{dom}(\mathfrak{b})$. To begin, we have
$$
2A=\left\{x_1^2+x_2^2 + x_3^2 +x_4^2: x_1,x_2, x_3, x_4 \in V\right\}.
$$
Fix $k \in \NNb^+$ and $h \in \NNb$. By Lagrange's four square theorem, there exist $y_1, y_2, y_3, y_4 \in \NNb$ such that $h = \allowbreak y_1^2+y_2^2+y_3^2+y_4^2$. Set, for each $i \in \llb 1, 4 \rrb$, $n_i := (h+1)k + y_i$ and $x_i := n_i! + n_i$, and note that $x_i \in V$, $x_i \ge h$, and $n_i \ge k$. It is then easily checked that
$$
\sum_{i=1}^4 x_i^2 \equiv \sum_{i=1}^4 (n_i! \, (n_i! + 2n_i) + n_i^2) \equiv \sum_{i=1}^4 n_i^2 \equiv \sum_{i=1}^4 y_i^2 \equiv h\bmod{k}.
$$
Therefore $(k\cdot \HHb+h) \cap 2A$ is non-empty and, since $k$ and $h$ were arbitrary, we conclude that the only arithmetic progression of $\bf H$ containing $2A$ is $\mathbf H$ itself, with the result that $\mathfrak{b}^\star(2A)=1$. 

Now suppose for a contradiction that $\mathfrak{b}_\star(2A)\neq 0$. By Eq.~\eqref{equ:lower-buck-density}, this is only possible if $2A$ contains an arithmetic progression of $\HHb$, implying that there is a constant $C \in \mathbf R^+$ such that $|2A \cap [1,m]| \ge Cm$ for all large $m$. The latter is, however, a contradiction, because it is clear that
$$
|2A \cap \llb 1,m \rrb| \le |V \cap \llb 1,\sqrt{m} \,\rrb|^4 \le \sup\{n^4: n \in \mathbf N \text{ and } n! \le \sqrt{m}\}=o(m),
\quad \text{as } m \to \infty.
$$
It follows that $\mathfrak{b}_\star(2A)=0 \ne \mathfrak b^\star(2A)$, and hence $2A\notin \mathrm{dom}(\mathfrak{b})$.

\section*{Acknowledgments}

P.\,L.~is grateful to PRIN 2017 (grant 2017CY2NCA) for financial support. Both authors thank the anonymous reviewers for a careful reading of the manuscript and many suggestions that helped to improve the overall quality of the paper.

\end{document}